\documentclass{amsart}
\usepackage{amsmath,amssymb,amsthm,eucal,graphics,dsfont}
\usepackage{graphicx}
\usepackage{amsfonts,amssymb,eucal}
\usepackage{amsbsy}
\usepackage{latexsym}
\usepackage[usenames]{color}
\usepackage{todonotes}
\usepackage{colortbl}

\usepackage{hyperref}

\usepackage[capitalize]{cleveref}
\usepackage{tikz-cd}

\newcommand{\mc}[1]{\mathcal{#1}}
\newcommand{\LR}{\Leftrightarrow}

\newcommand{\A}{\mathcal{A}}
\newcommand{\B}{\mathcal{B}}
\renewcommand{\phi}{\varphi}
\def\endref{\end{enumerate}\end{small} }
\newtheorem{theorem}{Theorem}

\newtheorem{definition}{Definition}
\newtheorem{proposition}{Proposition}
\newtheorem{lemma}{Lemma}

\newtheorem{corollary}{Corollary}

\begin{document}
\subjclass{03C57}
\keywords{computable structures, linear ordering, boolean algebra, computable categoricity,bi-embeddability}
\title{On bi-embeddable categoricity of algebraic structures}
\author{Nikolay Bazhenov}
\address{Sobolev Institute of Mathematics, 4 Acad. Koptyug Ave. and Novosibirsk State University, 2 Pirogova St., Novosibirsk, 630090, Russia}
\thanks{Bazhenov and Zubkov are supported by the RSF grant (project No. 18-11-00028).}
\email{bazhenov@math.nsc.ru}

\author{Dino Rossegger}
\address{Department of Pure Mathematics, University of Waterloo, 200 University Ave West, Waterloo, Ontario, Canada}
\email{dino.rossegger@uwaterloo.ca}

\author{Maxim Zubkov}
\address{Kazan Federal Univerity, 35 Kremlevskaya Str., Kazan, Russia}
\email{maxim.zubkov@kpfu.ru}
\begin{abstract}
In several classes of countable structures it is known that every hyperarithmetic structure has a computable presentation up to bi-embeddability. In this article we investigate the complexity of embeddings between bi-em\-bed\-dable structures in two such classes, the classes of linear orders and Boolean algebras. We show that if $\mc L$ is a computable linear order of Hausdorff rank $n$, then for every bi-embeddable copy of it there is an embedding computable in $2n-1$ jumps from the atomic diagrams. We furthermore show that this is the best one can do: Let $\mc L$ be a computable linear order of Hausdorff rank $n\geq 1$, then $\mathbf 0^{(2n-2)}$ does not compute embeddings between it and all its computable bi-embeddable copies. We obtain that for Boolean algebras which are not superatomic, there is no hyperarithmetic degree computing embeddings between all its computable bi-embeddable copies. On the other hand, if a computable Boolean algebra is superatomic, then there is a least computable ordinal $\alpha$ such that $\mathbf 0^{(\alpha)}$ computes embeddings between all its computable bi-embeddable copies. The main technique used in this proof is a new variation of Ash and Knight's pairs of structures theorem.
\end{abstract}

\maketitle

\sloppy
\rm

\section{Introduction}
	It is well-known that arbitrary isomorphic algebraic structures $\mathcal{A}$ and $\mathcal{B}$ possess the same algebraic properties. In contrast to this fact, the \emph{algorithmic properties} of $\mathcal{A}$ and $\mathcal{B}$ can be strikingly different. One of the first examples of this phenomenon was witnessed by Mal'tsev~\cite{Mal62}: While the standard recursive copy $\mathcal{A}$ of $\mathbb{Q}^{\omega}$ (i.e. the divisible torsion-free abelian group of countably infinite rank) has an algorithm for linear dependence, Mal'tsev built a copy $\mathcal{B}$ of $\mathbb{Q}^{\omega}$ with no such algorithm.

	These kinds of algorithmic discrepancies motivated a plethora of questions,  shaping modern computable structure theory. One of these questions can be (informally) stated as follows: What is the simplest possible presentation of a given structure $\mathcal{A}$? Or more concretely, given some countable structure $\mathcal{A}$, which is ``hard to compute'', is it possible to find a copy $\mathcal{B} \cong \mathcal{A}$ such that $\mathcal{B}$ is ``computationally more simple'' than $\mathcal{A}$ itself?

 Quite remarkably, one of the pioneering results of computable structure theory deals with the questions above: In 1955, Spector~\cite{spector1955} showed that every hyperarithmetic well-order is isomorphic to a computable one.

In turn, this classic result raises a new question~--- How hard is it to compute isomorphisms between well-orders? Ash~\cite{Ash86} answered this question by showing that if $\alpha$ is such that $\omega^{\delta+n} \leq \alpha < \omega^{\delta+n+1}$, where $\delta$ is either $0$ or a limit ordinal and $n\in \omega$, then for any two orderings $\A$ and $\B$ of order type $\alpha$, there is a $\Delta^{\A\oplus \B}_{\delta+2n}$ computable isomorphism, but there are orderings $\hat \A$ and $\hat \B$ of this order type such that for $\beta<\delta+2n$, $\Delta^{\A\oplus\B}_{\beta}$ does not compute any isomorphisms.
Recently, Alvir and Rossegger~\cite{alvir2018} generalized this result to arbitrary scattered linear orderings by giving precise bounds on the complexity of Scott sentences in this class. Note that in their result one loses effectiveness, however, obtaining bounds on the complexity of the isomorphisms in the Borel hierarchy.

By using the fact that any countable superatomic Boolean algebra is isomorphic to the interval algebra of an ordinal $Int(\alpha)$, one can obtain similar results to those of Ash for superatomic Boolean algebras~\cite{Ash87,AK00}. In general, there is a large body of literature dealing with the algorithmic complexity of isomorphisms: We refer the reader to the surveys \cite{Ash-98,FHM-14} and the monograph~\cite{AK00} for the results on Turing degrees of isomorphisms. For the subrecursive complexity of isomorphisms, the reader is referred to the surveys~\cite{CR-98,Mel-17,BDKM-19}.

One weakening of the notion of isomorphism is \emph{bi-embeddability} or, as it is sometimes called in the literature, \emph{equimorphism}:
\begin{definition}
  Two structures are \emph{bi-embeddable} if either is isomorphic to a substructure of the other.
\end{definition}
Given two structures $\A$ and $\B$ we write $\A \hookrightarrow \B$ to denote that $\A$ is embeddable in $\B$, i.e., that there is a 1-1 structure preserving map from $\A$ to $\B$. We also write $f:\A\hookrightarrow \B$ to say that $f:A\rightarrow B$ is an embedding of $\A$ in $\B$. If $\A\hookrightarrow \B$ and $\B\hookrightarrow \A$, i.e., if $\A$ and $\B$ are bi-embeddable, then we write $\A\approx \B$.


Montalb\'an~\cite{montalban2005} was able to obtain an analogue to Spector's result for all linear orders with respect to bi-embeddability.
\begin{theorem}[Montalb\'an]
  Every hyperarithmetic linear order is bi-embeddable with a computable one.
\end{theorem}
Greenberg and Montalb\'an subsequently obtained the same result for Boolean algebras, abelian groups, and compact metric spaces~\cite{GreenMon}. For the class of equivalence structures one can even obtain that every equivalence structure is bi-embeddable with a computable one~\cite{fokina2019a}. These remarkable results raise the question of how hard it is to compute embeddings between bi-embeddable structures.

In this article we obtain partial answers to these questions by calculating the complexity of embeddings between bi-embeddable countable Boolean algebras and linear orders of finite Hausdorff rank.

To do this we study the notions of relative bi-embeddable categoricity and degree of bi-embeddable categoricity. These notions are analogues of well studied notions for isomorphism, see~\cite{AKMS-89,Chi-90,fokina2010,CFS-2013,FHM-14}. 
\begin{definition}
  A countable (not necessarily computable) structure $\A$ is \emph{relatively $\Delta^0_\alpha$ bi-embeddably categorical} if for any bi-embeddable copy $\B$, $\A$ and $\B$ are bi-embeddable by  $\Delta^{\A \oplus \B}_\alpha$ embeddings. A computable structure is relatively computably bi-embeddably categorical if $\alpha=1$.
\end{definition}
\begin{definition}
  Let $\A$ be a computable structure. If $\mathbf d$ is the least Turing degree computing embeddings between any two computable structures bi-embeddable with $\A$, then $\mathbf d$ is the \emph{degree of bi-embeddable categoricity} of $\A$.
\end{definition}
Computable bi-embeddable categoricity and degrees of bi-embeddable categoricity were studied in a more general context in~\cite{bazhenov2020}. Notice that a structure may not have a degree of bi-embeddable categoricity. As we will discuss later in this article, even natural examples such as the order type of the rational numbers, $\eta$, and the atomless Boolean algebra do not have a degree of bi-embeddable categoricity. In related work, Bazhenov, Fokina, Rossegger, and San Mauro~\cite{bazhenov2018a} studied the complexity of embeddings between equivalence structures and showed that every computable equivalence relation has either degree of bi-embeddable categoricity $\mathbf 0$, $\mathbf 0'$ or $\mathbf 0''$.

Our main results show that linear orders of finite Hausdorff rank $n$ are relatively $\Delta^0_{2n+2}$ bi-embeddably categorical, but not relatively $\Delta^0_{2n+1}$, and that all computable superatomic Boolean algebras have a degree of bi-embeddable categoricity depending on their Frech\'et rank. We present the necessary preliminaries and our results about linear orders in \cref{sec:lin}. In \cref{sec:ba} we present our results about Boolean algebras. We give a short summary about $\alpha$-systems and then prove a variation of Ash and Knight's~\cite{AK90} pairs of structure theorem which we will use to calculate the degrees of bi-embeddable categoricity for superatomic Boolean algebras.

\section{Linear orders}\label{sec:lin}

\subsection{Preliminaries}

A linear order $\mc L$ is given by a pair $(L,\leq_{\mc L})$, where $L$ is a set and $\leq_{\mc L}$ is a binary relation on $L$ satisfying the usual axioms of linear orders. If $\mc L$ is infinite, then we assume that $L=\omega$.
Given $\mc L$, we let $<_\mc L$ be the induced strict ordering, i.e. for all $x,y\in L$, we have $x<_\mc L y$ if and only if $x\leq_\mc L y$ and $y\not\leq_\mc L x$. We will also use interval notation: $[x,y]_\mc L=\{ z: x\leq_\mc L z \leq_\mc L y \}$; and we make use of the following additional relations on linear orders.
\begin{itemize}
   \item The \emph{successor relation} $S_\mc L$ given by
   \[ S_\mc L(x,y) \ \LR\  x<_\mc L y \ \land\ \forall z (z\leq_\mc L x \lor y\leq_\mc L z),\]
   \item and the \emph{block relation} $F_\mc L(x,y)$ defined by
   \[ F_\mc L(x,y) \ \LR\ \begin{cases} [x,y]_\mc L \text{ is finite} &\text{ if }x\leq_\mc L y; \\ [y,x]_\mc L\text{ is finite}& \text{ if }y\leq_\mc L x.\end{cases}\]
\end{itemize}
We will drop subscripts if the order is clear from context.

It is not hard to see that the block relation is an equivalence relation on $\mc L$ which agrees with $\leq_\mc L$, and that it is definable by a computable $\Sigma^0_2$ formula in $L_{\omega_1,\omega}$. We call its equivalence classes \emph{blocks} and denote the block of $x\in L$ as $[x]_\mc L$. As the block relation agrees with the ordering, we can take the quotient structure and obtain the \emph{factor ordering} or \emph{condensation}. It is denoted by $\mc L/F_\mc L$ and defined as usual by $[x]_\mc L \leq_{\mc L/F_\mc L} [y]_\mc L \LR x\leq_\mc L y$.

Taking condensations can be iterated finitely often in an obvious way~--- by factoring through the block relation of the previous condensation. In order to obtain a notion of iterated condensation for all ordinals, we define the $\alpha$-block relation for all ordinals $\alpha$.

\begin{definition}
  Given a linear order $\mc L$, the \emph{$\alpha$-block relation $F_{\mc L}^\alpha$} on $\mc L$ is defined by induction as follows. Let $[x]_\mc L^\alpha$ denote the $F_{\mc L}^\alpha$-equivalence class of $x\in L$. Then for $x,y\in L$,
  \begin{enumerate}
    \item $F^0_\mc L(x,y)\LR x=y$,
    \item if $\alpha=\beta+1$, then $F^\alpha_\mc L(x,y)\LR F_{\mc L/F^\beta_\mc L}([x]^\beta_\mc L, [y]^\beta_\mc L)$, and
    \item if $\alpha$ is a limit ordinal, then $F^\alpha_\mc L(x,y)\LR (\exists \beta < \alpha) F_{\mc L/F^\beta_\mc L}(x,y)$.
  \end{enumerate}
\end{definition}

It is not hard to see that for a finite $\alpha$, $\mc L/F^\alpha_\mc L$ agrees with taking condensations iteratively $\alpha$ times, and that $F^\alpha_\mc L$ is $\Sigma^0_{2\alpha}$ definable. To simplify notation, we set $\mc L^{(\alpha)}_F=\mc L/F^\alpha_\mc L$.

\begin{definition}
  The \emph{Hausdorff rank} of a linear order $\mc L$ is the least $\alpha$ such that $\mc L^{(\alpha)}_F=\mc L^{(\alpha+1)}_F=1$.
\end{definition}

As usual, we will identify with $\omega$ the order type of the natural numbers, with $\zeta$ the order type of the integers, and with $\eta$ the order type of the rationals. The unique finite order type with precisely $n$ elements is denoted by $\mathbf n$. We let $\mathbf 0$ stand for the empty ordering. Further, if $\mc L$ is a linear order, then $\mc L^*$ is its reverse ordering, i.e., $x\leq_{\mc L^*} y\LR y\leq_\mc L x$.

\begin{definition}
  A linear order is \emph{scattered} if it has no subordering of order type $\eta$.
\end{definition}

\begin{definition}
  The class $\mathbf{VD}$ of linear orderings is defined by
  \begin{enumerate}
    \item ${\bf VD}_0=\{0,\,1\}$,
    \item ${\bf VD}_\alpha = \bigg\{ \sum\limits_{i\in\tau}\mathcal{L}_i \,\colon  \mathcal{L}_i\in \bigcup\limits_{\beta<\alpha}{\bf VD}_\beta,\,\tau\in \{\omega,\omega^*,\zeta\} \bigg\}$, and
    \item $\mathbf{VD}=\bigcup\limits_{\alpha}\bf{VD}_\alpha$.
  \end{enumerate}
  The \emph{$VD$-rank} of a linear order $\mc L$ is the least $\alpha$ such that $\mc L\in \mathbf{VD}_\alpha$; and the \emph{$VD^*$-rank} of $\mc L$ is the least $\alpha$ such that $\mc L$ is a finite sum of linear orders in $\mathbf{VD}_\alpha$.
\end{definition}

The following theorem due to Hausdorff is well known.
\begin{theorem}
  A countable linear order is scattered if and only if it has countable VD-rank. Furthermore, the $VD$-rank of a scattered linear order is equal to its Hausdorff rank.
\end{theorem}

A linear order $\mathcal{L}$ is \emph{indecomposable} if whenever $\mathcal{L}=\mathcal{A}+\mathcal{B}$, we have that $\mathcal{L}$ can be embedded in either $\mathcal{A}$ or $\mathcal{B}$. It is not hard to prove that a linear order bi-embeddable with an indecomposable one is also indecomposable.

A \emph{signed tree} is a pair $\langle T,s_T\rangle$, where $T$ is a well-founded subtree of $\omega^{<\omega}$ (i.e. a downwards closed subset of $\omega^{<\omega}$ with no infinite paths), and $s_T$ is a map, called \emph{sign function}, from $T$ to $\{+,-\}$. We will usually write $T$ instead of $\langle T,s_T\rangle$.
Given a signed tree $T$, let $T_{\sigma}$ be the subtree of $T$ with root $\sigma$.

We associate with every signed tree $T$ a linear order $lin(T)$ whose order type is defined by recursion as follows. If $T=\{\emptyset\}$ and $s_T(\emptyset)=+$, then $lin(T)=\omega$. If $T = \{\emptyset\}$ and $s_T(\emptyset)=-$, then $lin(T)=\omega^*$. Now, assume that $lin(T_{\sigma})$ has been defined for $\sigma\in T\setminus\{ \emptyset\}$. If $s_T(\emptyset)=+$, then
\[ lin(T)\cong \sum_{i\in \omega}\sum_{j\leq i} lin(T_{\langle j\rangle}).\]
If $s_T(\emptyset)=-$, then
\[ lin(T)\cong \sum_{i\in \omega^*}\sum_{j\leq i} lin(T_{\langle j\rangle}).\]

Linear orders of order type $lin(T)$ for some signed tree $T$ are called \emph{h-in\-de\-com\-po\-sable linear orders}. Montalb\'an~\cite{montalban2006} showed that every indecomposable linear order is bi-embeddable with an h-indecomposable linear order. Furthermore, for an indecomposable linear order $\mathcal{L}$ of finite Hausdorff rank, the rank of the signed tree associated with its bi-embeddable h-indecomposable order is equal to the Hausdorff rank of $\mathcal{L}$.


\subsection{Upper bounds}

In this section we give an upper bound on the complexity of embeddings between two bi-embeddable linear orders of finite Hausdorff rank.
%
Notice that there is only one non-scattered order type up to bi-embeddability~--- the ordering of the rational numbers, $\eta$. It is not relatively $\Delta^0_\alpha$ bi-embeddably categorical for \emph{any} computable ordinal $\alpha$. To see this, fix a standard copy of $\eta$ (in which one can compute an infinite decreasing sequence) and a copy of the Harrison ordering $H=\omega_1^{\mathrm{CK}}(1+\eta)$ without hyperarithmetic infinite decreasing sequences. Any embedding of $\eta$ into $H$ will compute an infinite decreasing sequence, so there can be no hyperarithmetic embedding.

Results of Montalb{\'a}n show that any hyperarithmetic linear order, and, moreover, any linear order of computable Hausdorff rank is bi-embeddable with a computable one. This section and the next one provide sharp bounds on the complexity of embeddings between linear orders of finite Hausdorff rank.

\begin{lemma}\label{uniformoperators}
  Fix $n\in \omega$. There are Turing operators $\Phi$ and $\Psi$ such that if $\A$ and $\B$ are of VD-rank $n$, $a,b\in \omega$, $\circ\in \{<,>,\leq,\geq,(,[\}$ and $m\leq n$
  \begin{enumerate}
    \item $\Phi^{(\A\oplus \B)^{(2m+1)}}(a,b,m,\circ)=\begin{cases} 1&\text{ if } [\circ a]^m\hookrightarrow[\circ b]^m \\
    0&\text{ otherwise}\end{cases}$,
    \item $\Psi^{(\A\oplus \B)^{(2m)}}(a,b,m,\circ,-)=\begin{cases}[\circ a]^m\hookrightarrow [\circ b]^m & \text{if } \Phi^{(\A\oplus \B)^{(2m+1)}}(a,b,m,\circ)=1\\
    \uparrow &\text{otherwise}
  \end{cases}$

  \end{enumerate}
  where $[\circ x]^m$ is the subordering on $\{y: y\circ x \ \&\ y\in [x]^m\}$ if $\circ\in\{<,>,\leq,\geq\}$ and
    $[(\ x]^m$ is the subordering on $\{y: y\in (x_1,x_2]\}$ if $x=\langle x_1,x_2\rangle$ with $x_1\in [x_2]^m$ and $[>x_1]$ if $x_1\not\in [x_2]^m$. Similarly, $[[\ x]^m$ is the subordering on $\{y: y\in [x_1,x_2]\}$ if $x=\langle x_1,x_2\rangle$ with $x_1\in [x_2]^m$ and $[>x_1]$ if $x_1\not\in [x_2]^m$.
\end{lemma}
 \begin{proof}
 The proof is by induction on $m$. Fix $a,b\in\omega$ and, as for $m=0$ the lemma trivializes, assume that $m=1$. Consider $[\circ a]^1$ and $[\circ b]^1$. The two orders can be of the following types: finite, $\omega$, $\omega^*$ or $\zeta$. It is easy to see that they are bi-embeddable if and only if they are isomorphic. We can determine whether the two orderings are isomorphic and subsequently define an embedding between them by checking whether they have a first and, or, last element, calculating their size in case they are finite, and calculating their successor relation. All of this can be done within three jumps over the diagrams, for example the following formula is satisfied by $\A$ if $[>a]^1$ has a greatest element:
 \begin{equation} \exists  y \forall x (x\in [a]^1 \rightarrow x\leq y)\label{form:greatestele}\end{equation}
 As the relation $x\in [a]^1$ is $\Sigma^0_2$, the formula is $\Pi^0_3$ and thus, using $(\A\oplus \B)^{(3)}$ as an oracle we can decide whether it is true or not.
 It is straightforward to define $\Phi$ by evaluating the formulas corresponding to the statements mentioned above so that $\Phi^{(\A\oplus \B)^{(3)}}(a,b,1,\circ)=1$ iff $[\circ a]^1\hookrightarrow [\circ b]^1$.

 Defining an embedding is even easier and only takes two jumps over the diagrams in case that $[\circ a]^1\hookrightarrow [\circ b]^1$. Assume $\circ$ equals $>$, the other cases are symmetric. As $[\circ a]^1$ must have a least element, the unique element $c$ satisfying the following formula:
 \[ c\in [a]^1\land c>a\land \forall y ((y\in [a]^1 \land y>a )\rightarrow y \geq c)\]
 This formula is a conjunction of a computable $\Sigma^0_2$ formula and a computable $\Pi^0_2$ formula and therefore within two jumps over $\A$ we can find the least element. Likewise, we can find the least element $d$ in $[\circ b]^1$. Now let $s_\A$, $s_\B$ be the successor function on $\A$ and $\B$ respectively, then we define $\Psi$ by
 \[\Psi^{(\A\oplus\B)^2}(a,b,1,>,x)=\begin{cases}
 y, & \exists k s_\A^k(c)=x \land s_\B^k(d)=y\\
 \uparrow, & \text{otherwise}
\end{cases}.\]
 It is not hard to see that $\Psi$ defined like this is a well-defined computable operator.
 %
 %
 %
 %
 %
 %
 %
 %
 %
 %
 %
 %
 %
 %

 Assume that the theorem holds for all $k<n$. We can establish the order type of $[\circ a]^n/F^{n-1}$ and $[\circ b]^n/F^{n-1}$ in similar fashion as in the base case by checking whether there are least and, or, greatest elements and the size of the orders in case they are finite. For instance to check whether $[>a]^n$ has a greatest element we only need to replace $x\in [a]^1$ by $x\in [a]^n$ in \cref{form:greatestele}. The resulting formula is $\Sigma^0_{2n+1}$ and thus in $2n+1$ jumps over the diagrams we can evaluate it and the formulas defining the other properties.

 This is however not enough to say that $[\circ a]^n\hookrightarrow [\circ b]^n$. We will describe in detail the case when $\circ$ is $\geq$ and $[\geq a]^n/F^{n-1}\cong[\geq b]^n/F^{n-1}\cong\omega$. The other cases are symmetric.
 We can pick the least natural number from each of the $(n-1)$-blocks 
 in $[\geq a]^n$, respectively $[\geq b]^n$ and write them in order, i.e.,
 \begin{align*}
   a_0\quad & a_1\quad & a_2\quad & a_3\quad & a_4\quad & a_5\quad & a_6\quad & a_7\quad & a_8\quad & a_9\quad &\dots\\
   b_0\quad & b_1\quad & b_2\quad & b_3\quad & b_4\quad & b_5\quad & b_6\quad & b_7\quad & b_8\quad & b_9\quad &\dots
 \end{align*}
 We have to find embeddings of the $(n-1)$-blocks of $a_i$ into $[\geq b]^n$ preserving order. We can do this inductively as follows:

 We have that $[\geq a]^{n}\hookrightarrow [\geq b]^n$ only if there are $m$, $a_{0}^0<\dots <a_0^k=a_0<\dots <a_0^m$ such that $a_0^i\in [a_0]^{n}$ for all $i$, and $b^0<\dots < b^m\in [b_j]^{n-1}$ for the least $j$ such that
 \[ [\leq a_0^0]^{n-1}\hookrightarrow [\leq b^0]^{n-1}\land (a_0^0,a_0^1] \hookrightarrow (b^0,b^1] \land \dots \land [>a_0^m]\hookrightarrow [>b^m]^{n-1},\]
 where all the half-open intervals are contained in an $(n-1)$-block.
 The existence of such a sequence follows from the facts that images under an embedding of elements from different $(n-1)$-blocks must be in different $(n-1)$-blocks, and that the image of an $(n-1)$-block must lie in a finite set of $(n-1)$-blocks.

 By hypothesis evaluating the above statement can be done uniformly in $(\A\oplus \B)^{(2(n-1)+2)}$. If the statement is true, then we proceed with $a_1$ in place of $a_0$ and by restricting our search for elements $b^i$ to elements in blocks greater than $[b_j]^{n-1}$. 
 If the statement is false, then we stop and conclude that $[\geq a]^n$ is not embeddable into $[\geq b]^n$.

 Now it is easy to see that $[\geq a]^n\hookrightarrow [\geq b]^n$ if and only if the above statement holds for all $a_i$, and to verify this uniformly we need another jump, i.e., $(\A\oplus \B)^{(2n+1)}$. Using this as an oracle we can define $\Phi$ as required.

 We can define $\Psi$ by induction in the same manner as we verified whether $[\geq a]^{n} \hookrightarrow [\geq b]^n$. First obtain the elements $a_0,\dots$ and $b_0,\dots$ as above. Then given $x$ determine the $a_i$ such that $x\in [a_i]^{n-1}$. First find the least $j$ satisfying the condition described above for $a_0$, then for $a_1$, and so on to obtain the interval of $[\geq b]^n$ in which $x$ embeds if such an interval exists. If it exists then use the hypothesis to define the embedding, otherwise stay undefined. 

 Given the arguments above this can be done from $(\A\oplus \B)^{(2(n-1)+2)}=(\A\oplus \B)^{(2n)}$.
\end{proof}

\begin{theorem}
Suppose that $\mathcal{L}$ is a scattered linear order of finite $VD^*$-rank $n$. Then it is relatively $\Delta^0_{2n}$ bi-embeddably categorical.
\end{theorem}
\begin{proof}

  First consider the case where $\mc L$ has $VD$-rank $n$. Every bi-embeddable copy of it must also have $VD$-rank $n$. Let $\A$, $\mc B$ be bi-embeddable with $\mc L$ and fix $a\in A$. Then there is an embedding $f\colon \A\hookrightarrow \B$ sending $a$ to $f(a)$. In particular, $[<a]^n\hookrightarrow [<f(a)]^n$ and $[\geq a]^n\hookrightarrow [\geq f(a)]^n$. We show how to construct a $\Delta^{\A\oplus\B}_{2n}$ embedding of $[\geq a]^n\hookrightarrow [\geq f(a)]^n$ in case that $[\geq a]^n/F^{n-1}\cong \omega$. It is easy to adapt the construction for the case when $[\geq a]^n/F^{n-1}\cong \mathbf m$ for some $m\in\omega$.

	We proceed similarly as in the proof of \cref{uniformoperators}. First we obtain a representation of $[\geq a]^n/F^{n-1}\cong \omega$ and $[\geq f(a)]^n/F^{n-1}$ using the least elements in each block. Thus we obtain ordered sequences $a_0,\dots$ and $b_0,\dots$. We then inductively define the embedding. The fact that $[\geq a]^n\hookrightarrow [\geq f(a)]^n$ implies that there are $m$, $a_{0}^0<\dots <a_0^k=a_0<\dots <a_0^m$ such that $a_0^i\in [a_0]^{n-1}$ for all $i$, and $b^0<\dots < b^m\in [b_j]^{n-1}$ for the least $j$ such that
 \[ [\leq a_0^0]^{n-1}\hookrightarrow [\leq b^0]\land (a_0^0,a_0^1] \hookrightarrow (b^0,b^1] \land \dots \land [>a_0^m]^{n-1}\hookrightarrow [>b^m]^{n-1},\]
 where all the half-open intervals are contained in an $(n-1)$-block. We pick the least $j$ such that for $[b_j]^{n-1}$ this condition holds. Now, obtaining the sequences $a_0,\dots$ and $b_0,\dots$ is $\Delta^{\A\oplus \B}_{2n}$ (the sequence is described by finite conjunctions of computable $\Pi^0_{2n-1}$ and $\Sigma^0_{2n-1}$ sentences). We then know that the condition is satisfied for $a_0$ and using $\Phi$ from \cref{uniformoperators} finding $m$ and the elements $a_0^i$ and $b^i$ is $\Delta^{\A\oplus\B}_{2(n-1)+2}=\Delta^{\A\oplus\B}_{2n}$. Using $\Psi$ from \cref{uniformoperators}, we can then define the embedding $a_0\hookrightarrow [b_0,b_j)$. Having defined an embedding of $[a_0]^{n-1}+\dots+ [a_i]^{n-1}$ into $[b_0,b_l)$ for some $l$, we can define the an embedding for $[a_{i+1}]^{n-1}$ similarly to the case of $a_0$ with the difference that we restrict our search to elements in blocks greater than the block of $b_l$.

 This clearly yields a $\Delta^{\A\oplus \B}_{2n}$ embedding of $[\geq a]$ into $[\geq f(a)]$. We can define a similar embedding to embed $[<a]$ into $[<f(a)]$ and thus also obtain a $\Delta^{\A\oplus \B}_{2n}$ embedding of $\A$ into $\B$.

 Now, say $\mathcal L$ has $VD^*$ rank $n$. Then every bi-embeddable copy of it must also have $VD^*$ rank $n$ and furthermore if $\A\approx \B\approx \mathcal L$ and $\A=\A_1+\dots +\A_n$ where each $\A_i$ is of $VD$ rank less or equal than $n$, then we have $\B=\B_1+\dots +\B_n$ where $VD(\B_i)=VD(\A_i)$ and $\A_i\approx \B_i$ for all $i\leq n$. Fix elements $a_i$ and $b_i$ from each of the $\A_i$, respectively $\B_i$. Then it is not hard to see that using the same strategy as in the case where $\mathcal L$ has $VD$ rank $n$, we can define an embedding. The only thing we need to change is that if $x\in [a_i]^n$, then we need to define the embedding for $x$ by only considering elements in $[b_i]^n$. This is again $\Delta^0_{2n}$.
\end{proof}


%
%
%
%
%

\subsection{Lower Bounds}

In this section we prove the following theorem.
\begin{theorem} \label{thm:lowerboundlo}
  Suppose that $\mathcal{L}$ is a computable scattered linear order of finite Hausdorff rank $n+1$.
  Then there are computable linear orders $\mathcal{G}$, $\mathcal{B}$ bi-embeddable with $\mathcal{L}$ such that there is no $\Delta^0_{2n+1}$-embedding of $\mathcal{G}$ into $\mathcal{B}$. 
\end{theorem}

\begin{lemma}\label{lem:lowerboundindec}
  Let $\mc L$ be a computable scattered indecomposable linear order of finite Hausdorff rank $n+1$. Then there are computable linear orders $\mathcal G$ and $\mc B$ bi-embeddable with $\mc L$ such that there is no $\Delta^0_{2n+1}$ embedding of $\mc G$ into $\mc B$.
\end{lemma}
\begin{proof}
  As $\mc L$ is indecomposable of rank $n+1$, there is an h-indecomposable linear order together with its signed tree $T$ of rank $n+1$. Given $T$ of rank $n+1$, let $\sigma\in T$ of length $n$ and let $P(\sigma)$ be the tree $\{ \tau: \tau \preceq \sigma\}$ with sign function inherited from $T$. Assume that $s_T(\emptyset)=+$, the case where $s_T(\emptyset)=-$ is analogous. Our ordering $\mathcal G$ is of order type $\sum_{i\in \omega}\left(\sum_{j\leq i} lin(T_{\langle j\rangle})+lin(P(\sigma))\right)$.
  Let $\mathcal G^n$ be a standard copy of $\omega$ with the elements labelled by the nodes of the tree of height $1$, i.e., we can write $\mathcal G^n$ as
  \[ \mathcal G^n=t_0+g_0+t_{0,1}+t_{1,1}+g_1+t_{0,2}+t_{1,2}+t_{2,2}+g_2+\dots.\]
  Clearly we can take $\mc G$ such that there is a computable function $\psi:\mc G\rightarrow \mc G^n$ taking $x$ in the $i^{th}$ copy of $lin(T_{\langle j\rangle})$ in $\mc G$ to $t_{j,i}$ and $x$ in the $i^{th}$ copy of $lin(P(\sigma))$ in $\mc G$ to $g_i$.
  %
  %

  We now build $\mc B$ such that no embedding $\mc B \hookrightarrow \mc G$ has degree $\Delta^0_{2n+1}$. We first build a $\Delta^0_{2n+1}$ computable linear order $\mc B^n$ such that $\mc B^n\cong \mc \mc G^n$ but that there is no $\Delta^0_{2n+1}$ embedding of $\mc B^n$ into $\mc G^n$. Towards that fix a listing $(\phi_e)_{e\in\omega}$ of all partial $\Delta^0_{2n+1}$ computable functions. We construct $\mc B^n$ in stages. At stage $0$, $\mc B^n$ is $\mc G^n$
  with the difference that every element $g_i$ is replaced by successive elements $b_{i,1}, b_{i,2}$, i.e., we can write $\mc B^n$ as
  \[ \mc B^n= t_0+b_{0,1}+b_{0,2}+t_{0,1}+t_{1,1}+b_{1,1}+b_{1,2}+t_{0,2}+\dots.\]
  We want to satisfy the requirements
  \[R_e:\quad  \phi_e: \mc B^n\not \hookrightarrow \mc G^n.\]

  At stage $s$, if for $e<s$, $s$ is the first stage greater than $e$ such that $\phi_{e,s}(b_{e,2})\!\downarrow\, =x$ for some $x$ and $g_k$ is least such that $x\geq g_k$, then put elements into $(b_{e,1},b_{e,2})$ such that $|(b_{e,1},b_{e,2})|> |[t_0,x]|$. This ensures that $\phi_e$ can not be an embedding. This finishes the construction.

  Note that in $\mc B^n$ 
  all the intervals $[b_{e,1},b_{e,2}]$ are finite with uniformly computable first and last element. Notice furthermore that everything except those intervals in the ordering $\mc B^n$ is computable. Thus if we replaced those intervals with computable intervals we would get a computable ordering.

  We now construct for every $i<n$ a $\Delta^0_{2i+1}$ computable ordering $\mc B^i$ such that $\mc B^0=\mc B\approx \mc G$ and that there are $\Delta^0_{2i+1}$ computable embeddings from $\mc B^i$ into $\mc B^{i-1}$.

  We use a well known result that says if a linear order $\mc L$ is $\Delta^0_3$, then $\omega\cdot \mc L$ and $\omega^*\cdot \mc L$ are computable, see~\cite[Theorem 9.11]{AK00}. This theorem relativizes and the proof is constructive in the sense that it constructs a computable copy of $\{\omega,\omega^*\}\cdot \mc L$, given $\Delta^0_3$ order $\mc L$. Its only caveat is that it is nonuniform. If $\mc L$ has a least element, this element has to be fixed non-uniformly. However, this poses no problem to our construction, as we wish to jump-invert the intervals $[b_{i,1},b_{i,2}]$ and have uniformly computable least and greatest elements for those.

  So we take our $\Delta^0_{2n+1}$ computable linear order $\mc B^n$ and replace every interval $[b_{i,1},b_{i,2}]$ by a copy of $\omega\cdot[b_{i,1},b_{i,2}]$ if $\sigma(0)=+$ and by a copy of $\omega^*\cdot [b_{i,1},b_{i,2}]$ otherwise. Elements labelled $t_{i,j}$ are replaced by computable disjoint copies of $T_{\langle i\rangle}$. We obtain a linear order $\mc B^{n-1}$ with suborderings $\mc B^{n-1}_i$ corresponding to the jump inverted intervals $[b_{i,1},b_{i,2}]$. If $\mc B^{n-1}_i$ has a least or greatest element, then the procedure will also return us these elements. Notice that if for example $[b_{i,1}, b_{i,1}]\cong \mathbf n$ and $\sigma(0)=+$, then $\mc B^{n-1}_i\cong \omega\cdot \mathbf n$. Furthermore, $\mc B^{n-1}$ is a computable sum of uniformly computable and uniformly 
   $\Delta^0_{2n-1}$ linear orders and thus itself $\Delta^0_{2n-1}$. We repeat this procedure, replacing $\mc B^j_i$ with $\omega\cdot \mc B^j_i$ or $\omega^{*}\cdot \mc B^j_i$ depending on whether $\sigma(j)=+$ or $\sigma(j)=-$. The resulting linear order $\mc B^j$ is $\Delta^0_{2n+1-2(n-j)}$ computable. Thus by induction we end up with a linear order $\mc B^0=\mc B$ which is computable, and it is easy to see $\mc B\cong \sum_{i\in \omega}\left(\sum_{j\leq i} lin(T_{\langle j\rangle})+lin(P(\sigma))\cdot k_i\right)$ where $k_i=|[b_{i,1},b_{i,2}]|$. Clearly $\mc B \approx\mc G$.

  Our construction of $\mc B^i$ from $\mc B^{i+1}$ also provides us with $\Delta^{0}_{i}$ computable embeddings $\phi_{i}: \mc B^{i+1}\rightarrow \mc B^{i}$. Assume that there is a $\Delta^0_{2n+1}$ embedding $\chi$ of $\mc B$ into $\mc G$, then as can be seen from \cref{fig:compositions} the composition of the embeddings $\phi_{i}$ and $\psi$ gives a $\Delta^{0}_{2n+1}$ embedding of $\mc B^n$ into $\mc G^n$, a contradiction.

  \begin{figure}[h]
    \begin{center}
\begin{tikzcd}
\omega\cong\mathcal{B}^n  \arrow{rr}{\varphi_{n-1}} && \mathcal{B}^{n-1}  \arrow[dotted]{rr} && \mathcal{B}^{1}  \arrow{rr}{\varphi_{0}}&& \mathcal{B} \arrow{d}{\chi}\\
\omega\cong \mathcal{G}^n&&   &&  &&  \mathcal{G} \arrow{llllll}{\psi}
\end{tikzcd}
\caption{\label{fig:compositions}Morphisms between $\mc B$ and $\mc G$}
\end{center}
\end{figure}
\end{proof}

Jullien~\cite{jullien1969} showed that every countable scattered linear order can be decomposed into finitely many indecomposable linear orders and that there exists a minimal decomposition which is unique up to bi-embeddability, see also~\cite{montalban2006}.

Notice that if $\mc L$ has Hausdorff rank $\alpha$ and a minimal decomposition into indecomposable linear orders $\mc L_1+\dots+\mc L_n$, then at least one of the $\mc L_i$ must have rank $\alpha$. Exploiting this properties, we can prove \cref{thm:lowerboundlo}.

\begin{proof}[Proof of \cref{thm:lowerboundlo}]
Let $\mc L$ have Hausdorff rank $\alpha$ and $\mc L_1+\dots +\mc L_n$ be a minimal decomposition, where $\mc L_i$ is of rank $\alpha$. We take $\mc G$ to be the linear order $lin(T_0)+\dots+\mc G_i+\dots+lin(T_n)$ where $lin(T_k)$ is a computable h-indecomposable linear order bi-embeddable with $\mc L_k$ and $\mc G_i$ is the linear order $\mc G$ from \cref{lem:lowerboundindec} for $\mc L_i$.
From now on we will write $\mc G$ as the sum $\mc G_1+\mc G_i+\mc G_2$.

For $\mc B$ we use the linear order $\mc G_1+\mc B_i+\mc G_2$ where $\mc G_1$ and $\mc G_2$ are as for $\mc G$ and $\mc B_i$ is the linear order $\mc B$ from \cref{lem:lowerboundindec} for $\mc L_i$. We claim that no embedding of $\mc B$ into $\mc G$ can be $\Delta^0_{2n+1}$.

To prove this we first construct a computable labelling of $\mc G$ by taking our labelling $\psi$ from \cref{lem:lowerboundindec} and extending it to a labelling of $\mc G$ by labelling all elements in $\mc G_1$ and $\mc G_2$ with two new labels $h_1,h_2$. We also can canonically extend the embeddings $\phi^i$ constructed in the proof of \cref{lem:lowerboundindec} to embeddings of $\mc G_1+\mc B^{i+1}+\mc G_2$ into $\mc G_1+\mc B^{i}+\mc G_2$.

Let $\chi$ be an embedding of $\B$ into $\mc G$ and let $b_{i,2}^\phi$ be images of $b_{i,2}$ in $\mc B^n$ under $\phi^n\circ\dots \circ \phi^1$. Then, it is not hard to see, that if $\chi$ sends infinitely many elements $b_{i,2}^\phi$ to $\mc G_1$, then $\mc G_1+\mc B_i$ is bi-embeddable with $\mc G_1$, contradicting the minimality of the decomposition. On the other hand if there is one $b_{i,2}^\phi$ that goes to $\mc G_3$, then cofinitely many elements $b_{i,2}^\phi$ must go to $\mc G_3$. As $\mc B_i$ is bi-embeddable with any of its end segments, this implies that $\mc B_i+\mc G_3$ is bi-embeddable with $\mc G_3$~--- again contradicting the minimality of the decomposition.

Now, if $\chi$ was $\Delta^0_{2n+1}$, then $\phi^n\circ\dots\circ \phi^1\circ \chi \circ \psi$ would embed an end segment of $\mc B^n$ into $\mathcal N_\mathcal G$ and we can get a $\Delta^0_{2n+1}$ embedding by shifting the embedding to the right by a finite number $m$.
\end{proof}

\section{Boolean algebras}\label{sec:ba}

The reader is referred to, e.g., monographs~\cite{Gon-97,Koppelberg} for the background on countable Boolean algebras. We treat Boolean algebras as structures in the language $\{ \cup^2, \cap^2, \overline{(\ )}^1; 0,1\}$. Any Boolean algebra $\mathcal{B}$ admits a natural partial ordering: $x\leq_{\mathcal{B}} y$ iff $x\cup y = y$. We always assume that $0^{\mathcal{B}} \neq 1^{\mathcal{B}}$.


Suppose that $\mathcal{L}$ is a linear order with least element. Then the corresponding \emph{interval Boolean algebra} $Int(\mathcal{L})$ is defined as follows:
\begin{itemize}
	\item the domain of $Int(\mathcal{L})$ is the smallest set containing all finite unions of $\mathcal{L}$-intervals:
	\[
		[a_0,b_0) \cup [a_1,b_1) \cup \dots \cup [a_n,b_n) \text{ or } [a_0,b_0) \cup [a_1,b_1) \cup \dots  \cup [a_n,\infty)
	\]
	where $a_0 <_{\mathcal{L}} b_0 <_{\mathcal{L}} a_1 <_{\mathcal{L}} b_1 <_{\mathcal{L}} \dots <_{\mathcal{L}} a_n <_{\mathcal{L}} b_n$;

	\item the functions of $Int(\mathcal{L})$ are the usual set-theoretic operations.
\end{itemize}
For more details, see, e.g., Section~15 of~\cite{Koppelberg}.

Let $\mathcal{B}$ be a Boolean algebra. An element $a\in\mathcal{B}$ is an \emph{atom} if $a$ is a minimal non-zero element in $\mathcal{B}$. The algebra $\mathcal{B}$ is \emph{atomic} if for every non-zero $b\in\mathcal{B}$, there is an atom $a$ such that $a\leq_{\mathcal{B}} b$. The algebra $\mathcal{B}$ is \emph{superatomic} if all subalgebras of $\mathcal{B}$ are atomic.

The following fact is well-known: A countable Boolean algebra $\mathcal{B}$ is su\-per\-a\-to\-mic if and only if there are a countable ordinal $\alpha$ and a non-zero natural number $m$ such that $\mathcal{B}\cong Int(\omega^{\alpha}\cdot m)$ (see Theorem~1 of~\cite{Gon-73} or p.~277 of~\cite{Koppelberg}).

Furthermore, this superatomicity criterion admits a natural ``effectivization'': Goncharov (Theorem~2 of~\cite{Gon-73}) proved that a superatomic Boolean algebra $\mathcal{B}$ has a computable copy if and only if $\mathcal{B}\cong Int(\omega^{\alpha}\cdot m)$, where $\alpha$ is a computable ordinal and $0<m<\omega$.


The main result of this section is the following

\begin{theorem} \label{theo:BA}
	Let $\alpha$ be a non-zero computable ordinal, and let $m$ be a non-zero natural number. The superatomic Boolean algebra $Int(\omega^{\alpha} \cdot m)$ has degree of b.e.~categoricity equal to
	\[
		\mathbf{0}^{\langle 2\alpha\rangle} = \begin{cases}
 			\mathbf{0}^{(2\alpha-1)}, & \text{if } \alpha < \omega,\\
 			\mathbf{0}^{(2\alpha)}, & \text{if } \alpha \geq \omega.
		\end{cases}
	\]
\end{theorem}

Theorem~\ref{theo:BA} and Theorem~2.1 of~\cite{Bazh-13} together imply that for a computable superatomic Boolean algebra, its degree of b.e.~categoricity is \emph{equal} to its degree of categoricity. Nevertheless, we note that the proofs of~\cite{Bazh-13} cannot be directly transferred to the bi-embeddability setting: the key tool of this section is a new technique which employs limitwise monotonic functions (see Subsection~\ref{subsect:lim-mon}).

As a consequence of Theorem~\ref{theo:BA}, we obtain a complete description of degrees of b.e.~categoricity for Boolean algebras:

\begin{corollary}\label{corollary:gen-BA}
	Let $\mathcal{B}$ be a computable Boolean algebra. Then $\mathcal{B}$ satisfies one of the following two conditions:
	\begin{itemize}
		\item[(a)] There is a computable ordinal $\alpha$ such that $\mathbf{0}^{(\alpha)}$ is the degree of b.e.~categoricity for $\mathcal{B}$.

		\item[(b)] $\mathcal{B}$ is not hyperarithmetically b.e.~categorical, and $\mathcal{B}$ does not have degree of b.e.~categoricity.
	\end{itemize}
\end{corollary}

The further discourse is arranged as follows. Subsection~\ref{subsect:pairs} contains necessary preliminaries on the technique of pairs of structures, developed by Ash and Knight~\cite{AK90,AK00}. Furthermore, we give a new theorem, which allows to encode some specific families of limitwise monotonic functions into families of computable structures (Theorem~\ref{theo:uniform}). Since the proof of Theorem~\ref{theo:uniform} contains a lot of bulky technical details, this proof is delegated to the last subsection (Subsection~\ref{subsect:proof-uniform}).

Subsection~\ref{subsect:lim-mon} discusses some auxiliary facts about limitwise monotonic functions: in particular, Proposition~\ref{prop:good-lm} shows how one can encode the oracle $\mathbf{0}^{(\alpha)}$, where $\alpha < \omega^{CK}_1$, in a ``sufficiently good'' limitwise monotonic fashion. Subsection~\ref{subsect:proof-BA} proves Theorem~\ref{theo:BA}, and Subsection~\ref{subsect:coroll} obtains Corollary~\ref{corollary:gen-BA}.


\subsection{Pairs of computable structures} \label{subsect:pairs}

Let $\mathbf{d}$ be a Turing degree. A function $F\colon \omega\to\omega$ is \emph{$\mathbf{d}$-limitwise monotonic} if there is a $\mathbf{d}$-computable function $f\colon \omega\times\omega \to \omega$ such that:
\begin{itemize}
	\item[(a)] $f(x,s) \leq f(x,s+1)$ for all $x$ and $s$;

	\item[(b)] $F(x) = \lim_s f(x,s)$ for all $x$.
\end{itemize}


As per usual, we fix a path through Kleene's $\mathcal{O}$, and we identify computable ordinals with their notations along this path.

Let $\alpha$ be a non-zero computable ordinal. For the sake of convenience, we use the following notation:
\[
	\mathbf{0}^{\langle \alpha\rangle} := \begin{cases}
		\mathbf{0}^{(\alpha-1)}, & \text{if } \alpha < \omega,\\
		\mathbf{0}^{(\alpha)}, & \text{if } \alpha \geq \omega.
	\end{cases}
\]


For a language $L$, \emph{infinitary formulas} of $L$ are formulas of the logic $L_{\omega_1,\omega}$. For a countable ordinal $\alpha$, infinitary $\Sigma_{\alpha}$ and $\Pi_{\alpha}$ formulas are defined in a standard way (see, e.g., \cite[Chapter~6]{AK00}).

Suppose that $\mathcal{A}$ and $\mathcal{B}$ are $L$-structures, and $\alpha$ is a countable ordinal. We say that $\mathcal{A}\leq_{\alpha} \mathcal{B}$ if every infinitary $\Pi_{\alpha}$ sentence true in $\mathcal{A}$ is true in $\mathcal{B}$. The relations $\leq_{\alpha}$ are called \emph{standard back-and-forth relations}.

Ash~\cite{Ash86} provided a complete description of standard back-and-forth relations for countable well-orders~--- here is a small excerpt from the description:

\begin{lemma}[Lemma~7 of~{\cite{Ash86}}; see also Lemma~15.10 of~\cite{AK00}] \label{lem:Ash-001}
	Let $\beta$ be a countable ordinal, and let $k$ be a non-zero natural number. Then $\omega^{\beta} \cdot (k+1) \leq_{2\beta+1} \omega^{\beta} \cdot k$ and $\omega^{\beta} \cdot k \nleq_{2\beta+1} \omega^{\beta} \cdot (k+1)$.
\end{lemma}

Let $\alpha$ be a non-zero computable ordinal. A family $K = \{ \mathcal{A}_i\,\colon i\in \omega\}$ of $L$-structures is called \emph{$\alpha$-friendly} if the structures $\mathcal{A}_i$ are uniformly computable, and the relations
\[
		B_{\beta} = \big\{ (i,\bar a,j,\bar b) \,\colon\ i,j\in \omega,\, \bar a \textnormal{~is from~} \mathcal{A}_i,\, \bar b\textnormal{~is from~} \mathcal{A}_j,\,
		  (\mathcal{A}_i,\bar a)\leq_{\beta} (\mathcal{A}_j,\bar b) \big\}
\]
are c.e., uniformly in $\beta<\alpha$.


\begin{theorem}[Theorem 4.2 of~\cite{Bazh17}]\label{theo:Bazh-17}
	Let $\alpha$ be a non-zero computable ordinal. Suppose that $\{ \mathcal{A}_k\,\colon k\in\omega\}$ is an $\alpha$-friendly family of $L$-structures such that $\mathcal{A}_{k+1} \leq_{\alpha} \mathcal{A}_k$ for all $k\in\omega$. Then for any $\mathbf{0}^{\langle\alpha\rangle}$-limitwise monotonic function $g(x)$, there is a uniformly computable sequence of structures $(\mathcal{C}_n)_{n\in\omega}$ such that for every $n$, we have $\mathcal{C}_{n} \cong \mathcal{A}_{g(n)}$.
\end{theorem}

Theorem~\ref{theo:Bazh-17} admits the following generalization:

\begin{theorem}\label{theo:uniform}
	Let $\alpha$ be a computable infinite ordinal. Suppose that $(\alpha_i)_{i\in\omega}$	is a computable sequence of non-zero ordinals such that $\alpha_i < \alpha$ for all $i$. Suppose that $\{ \mathcal{A}^i_k\}_{i,k\in\omega}$ is an $\alpha$-friendly family of $L$-structures such that $\mathcal{A}^i_{k+1} \leq_{\alpha_i} \mathcal{A}^i_k$ for all $i$ and $k$. Consider a family of functions $(f_i)_{i\in\omega}$ such that each $f_i$ is $\mathbf{0}^{\langle \alpha_i\rangle}$-li\-mit\-wise monotonic, uniformly in $i\in\omega$, i.e. there is an index $e\in\omega$ such that for every $i$, the binary function $\Phi_e^{\emptyset^{\langle \alpha_i\rangle}}(i,\cdot,\cdot)$ approximates $f_i$ in a limitwise monotonic fashion. Then there is a uniformly computable sequence of structures $(\mathcal{C}^i_n)_{i,n\in\omega}$ such that $\mathcal{C}^i_n \cong \mathcal{A}^i_{f_i(n)}$ for all $i$ and $n$.
\end{theorem}

Recall that the proof of Theorem~\ref{theo:uniform} is given in Subsection~\ref{subsect:proof-uniform}.


\subsection{Limitwise monotonic functions} \label{subsect:lim-mon}

In this subsection, we give two useful facts about limitwise monotonicity. Recall that a partial function $\xi(x)$ \emph{dominates} a partial function $\psi(x)$ if
\[
	(\forall^{\infty} x) [ \psi(x)\!\downarrow\ \ \Rightarrow\  \xi(x)\!\downarrow\ > \psi(x)].
\]

\begin{lemma}[follows from Theorem 4.5.4 of \cite{Soare-New}]\label{lem:Soare}
	Let $A$ be an oracle.
	If a total function $g(x)$ dominates every partial $A$-computable function, then $A'\leq_T g\oplus A$.
\end{lemma}
\begin{proof}
	Fix an $A$-effective enumeration $\{ A'[s] \}_{s\in\omega}$ of the set $A'$, and consider the following partial $A$-computable function:
	\[
		\psi(x) := \begin{cases}
			\text{the least } s \text{  such that } x\in A'[s], & \text{if } x\in A',\\
			\text{undefined}, & \text{if } x\not\in A'.
		\end{cases}
	\]
	Since the function $g$ dominates $\psi$, there is a number $e$ such that:
	\[
		(\forall x \geq e) ( x\in A'\ \Leftrightarrow\ x\in A'[g(x)] ).
	\]
	Therefore, $A' \leq_T g\oplus A$.
\end{proof}

\begin{proposition}\label{prop:good-lm}
	Let $\alpha$ be a computable ordinal. There is a total function $f_{[\alpha]}(x)$ with the following property: If a total function $g$ dominates $f_{[\alpha]}$, then $g\geq_T \mathbf{0}^{(\alpha)}$. Furthermore:
	\begin{itemize}
		\item[(a)] If $\alpha = \beta+1$, then $f_{[\alpha]}$ is $\mathbf{0}^{(\beta)}$-limitwise monotonic.

		\item[(b)] If $\alpha$ is a limit ordinal, then $f_{[\alpha]}$ is $\mathbf{0}^{(\alpha)}$-computable.
	\end{itemize}
\end{proposition}
\begin{proof}
	Our proof employs recursion on $\alpha$.

	Without loss of generality, we may assume that for any oracle $A$, the partial computable function $\varphi^A_0$ satisfies the following:
	\[
		\varphi_0^A(x) := \begin{cases}
			\text{the least } s \text{ such that } \varphi_{x,s}^{A}(x)\!\downarrow, & \text{if such } s \text{ exists},\\
			\text{undefined}, & \text{otherwise}.
		\end{cases}
	\]

	We fix a Turing operator $\Psi$ with the following property: Let $\delta$ be a computable limit ordinal, and let $(\gamma_i)_{i\in\omega}$ be the standard fundamental sequence for $\delta$ (this sequence is induced by the notation of $\delta$ along our path through $\mathcal{O}$). Then we have
	\[
		\emptyset^{(\delta)} =\Psi^X\text{ where }X=\{ \langle i, x\rangle \,\colon i\in\omega,\, x\in\emptyset^{(\gamma_i)} \}.
	\]

	In order to obtain the desired functions, we simultaneously define $f_{[\alpha]}$ and a Turing operator $\Phi(\alpha;\cdot)$ satisfying the following condition: For any $\alpha$ and any $g$, if $g(x) > f_{[\alpha]}(x)$ \emph{for all} numbers $x$, then we have $\emptyset^{(\alpha)} = \Phi^{g}(\alpha;\cdot)$.

	\textsc{Case $\alpha = 0$.} We define $f_{[0]}(x):=0$ and $\Phi(0;\cdot):=0$.

	\textsc{Case $\alpha = \beta+1$.} Assume that the desired objects $f_{[\beta]}$ and $\Phi(\beta;\cdot)$ have been already defined. If $\beta = \gamma + 1$, then $f_{[\beta]}$ is $\mathbf{0}^{(\gamma)}$-limitwise monotonic. Thus, we deduce that $f_{[\beta]} \leq_T \mathbf{0}^{(\beta)}$. We set
	\begin{align*}
		f_{[\beta+1]}(2x) &:= f_{[\beta]}(x),\\
		f_{[\beta+1]}(2x+1) &:= 1 + \sum_{i\leq x} \{ \varphi^{\emptyset^{(\beta)}}_i(x)\,\colon \varphi^{\emptyset^{(\beta)}}_i(x) \! \downarrow\}.
	\end{align*}
	It is not hard to show that the function $f_{[\beta+1]}$ is $\mathbf{0}^{(\beta)}$-limitwise monotonic.

	Assume that a total function $g$ dominates $f_{[\beta+1]}$. Then the function $\tilde{g}(x):=g(2x)$ dominates $f_{[\beta]}$, and hence, we have $g\geq_T \tilde{g} \geq_T \mathbf{0}^{(\beta)}$. On the other hand, the function $\widehat{g}(x):=g(2x+1)$ dominates every partial $\mathbf{0}^{(\beta)}$-com\-pu\-table function $\varphi_i^{\emptyset^{(\beta)}}$. Therefore, by Lemma~\ref{lem:Soare}, $\widehat{g}\oplus \mathbf{0}^{(\beta)}$ computes $\mathbf{0}^{(\beta+1)}$. Thus, we obtain:
	\[
		g\equiv_T g\oplus \mathbf{0}^{(\beta)} \geq_T \widehat{g}\oplus \mathbf{0}^{(\beta)} \geq_{T} \mathbf{0}^{(\beta+1)}.
	\]

	Now we describe how to define $\Phi(\beta+1;\cdot)$. For the sake of simplicity, we assume that $g(x) > f_{[\beta+1]}(x)$ for all $x$, and we discuss the definition of $\Phi^g(\beta+1;\cdot)$ in this case. Since $\tilde{g}(y) > f_{[\beta]}(y)$ for all $y$, the function $\Phi^{\tilde{g}}(\beta;\cdot)$ is equal to the characteristic function of $\emptyset^{(\beta)}$.

	Note that for any $y$, if the value $\varphi^{\emptyset^{(\beta)}}_0(y)$ is defined, then $\widehat{g}(y) > \varphi^{\emptyset^{(\beta)}}_0(y)$. Our agreement about the partial function $\varphi_0^A$ implies that for all $y\in\omega$, we have
	\[
		y \in \emptyset^{(\beta+1)}\ \Leftrightarrow\ \varphi^{\emptyset^{(\beta)}}_{y,\widehat{g}(y)}(y)\!\downarrow\!.
	\]
	Since $\emptyset^{(\beta)} = \Phi^{\tilde{g}}(\beta;\cdot)$, it is now clear how one can obtain the desired equality $\Phi^g(\beta+1;\cdot ) = \emptyset^{(\beta+1)}$.

	\textsc{Case of limit $\alpha$.} Consider the standard fundamental sequence $(\beta_i)_{i\in\omega}$ for the ordinal $\alpha$. The function $f_{[\alpha]}$ is defined as follows:
	\[
		f_{[\alpha]}(\langle i,x\rangle) := f_{[\beta_i]}(x).
	\]
	By recursion on $\alpha$, it is not hard to show that $f_{[\alpha]}$ is $\mathbf{0}^{(\alpha)}$-computable.

	Assume that a total function $g$ dominates $f_{[\alpha]}$. Consider the functions $g_{i}(x):= g(\langle i,x\rangle)$, $i\in\omega$. Since $g$ dominates $f_{[\alpha]}$, there is a number $j_0$ such that
	\[
		(\forall i\geq j_0) (\forall x) (g_i(x) > f_{[\beta_i]}(x)).
	\]
	Hence, for every $i\geq j_0$, we have $\emptyset^{(\beta_i)} = \Phi^{g_i}(\beta_i;\cdot)$. Therefore, we deduce
	\[
		\mathbf{0}^{(\alpha)} \equiv_T \{ \langle i,y\rangle \,\colon i\geq j_0,\ y\in \emptyset^{(\beta_i)}\} \leq_T \{ \langle i, x, g_i(x)\rangle \,\colon i\geq j_0,\ x\in\omega\} \leq_T g.
	\]

	The definition of the object $\Phi(\alpha;\cdot)$ can be easily recovered from the following observation: If $g(x) > f_{[\alpha]}(x)$ for all $x$, then
	\[
		\emptyset^{(\alpha)} = \Psi^X \text{ where } X=\{ \langle i, x\rangle \,\colon i\in\omega,\, x\in\emptyset^{(\beta_i)}\}, \text{ and } \emptyset^{(\beta_i)} = \Phi^{g_i}(\beta_i;\cdot) \text{ for all } i\in\omega.
	\]
	Proposition~\ref{prop:good-lm} is proved.
\end{proof}

\subsection{Proof of Theorem~\ref{theo:BA}} \label{subsect:proof-BA}

	Let $\mathcal{B}:=Int(\omega^{\alpha}\cdot m)$. Using a standard Cantor--Bendixson analysis (see, e.g., Section~17.2 of~\cite{Koppelberg}), one can show the following: If a superatomic Boolean algebra $\mathcal{A}$ is bi-embeddable with $\mathcal{B}$, then $\mathcal{B}$ and $\mathcal{A}$ are isomorphic.

	For the interested reader, we note that in Corollary~4.34 of~\cite{GreenMon}, the following fact is proved: For countable ordinals $\beta$ and $\gamma$, we have $\beta \leq \gamma$ if and only if $Int(\omega^{\beta})$ is embeddable into $Int(\omega^{\gamma})$.

	Ash (Theorem~5 in~\cite{Ash87}) proved that the structure $\mathcal{B}$ is relatively $\Delta^0_{2\alpha}$ categorical. Therefore, it is sufficient to provide two computable isomorphic copies $\mathcal{A}$ and $\mathcal{C}$ of $\mathcal{B}$ such that every isomorphic embedding $h\colon \mathcal{A} \hookrightarrow \mathcal{C}$ computes the degree $\mathbf{0}^{\langle 2\alpha\rangle}$.

	Choose a computable structure $\mathcal{C}\cong\mathcal{B}$ with the following property: Given an element $b\in\mathcal{C}$, one can effectively compute a pair $\theta(b):=(\beta, \ell)$ such that $\beta \leq \alpha$, $\ell \in \omega$, and the relative algebra $\mathcal{C}\upharpoonright b$ is isomorphic to $Int(\omega^{\beta} \cdot \ell)$. Such a copy $\mathcal{C}$ was constructed, e.g., in Proposition~15.15 of~\cite{AK00}; see~\cite{Bazh16} for a more detailed discussion.

	\textsc{Case I.} Let $\alpha$ be a successor ordinal, i.e. $\alpha = \beta+1$.

	Choose a $\mathbf{0}^{\langle 2\beta+1\rangle}$-limitwise monotonic function $f(x)$ from Proposition~\ref{prop:good-lm} satisfying the following: if a total function $g(x)$ dominates $f$, then $g\geq_T \mathbf{0}^{\langle 2\beta+2\rangle} = \mathbf{0}^{\langle 2\alpha\rangle}$.

	By Lemma~\ref{lem:Ash-001}, we have $\omega^{\beta}\cdot (k+1) \leq_{2\beta+1} \omega^{\beta}\cdot k$ for every non-zero $k\in\omega$. Furthermore, the family $\{\omega^{\beta}\cdot k\,\colon k\geq 1\}$ is $(2\beta+1)$-friendly (see, e.g., Proposition~15.11 in~\cite{AK00}). Therefore, one can apply Theorem~\ref{theo:Bazh-17} and produce a computable sequence of linear orders $(\mathcal{L}_n)_{n\in\omega}$ such that
	\[
		\mathcal{L}_n \cong \omega^{\beta} \cdot (1+f(n)).
	\]

	Consider computable Boolean algebras
	\begin{align*}
		\mathcal{D}_j &:= \sum_{n\in\omega} Int(\mathcal{L}_n),\ \text{where}\ 1\leq j\leq m,\\
		\mathcal{A} &:= \sum_{1\leq j \leq m} \mathcal{D}_j.
	\end{align*}
	For $n\in\omega$, let $e_n$ be the element induced by $\mathbf{1}^{Int(\mathcal{L}_n)}$ inside $\mathcal{D}_1$. It is not hard to show that every $\mathcal{D}_j$ is isomorphic to $Int(\omega^{\beta+1})$, and $\mathcal{A}\cong Int(\omega^{\beta+1}\cdot m)$.

	Assume that $h$ is an isomorphic embedding from $\mathcal{A}$ into $\mathcal{C}$. Then for every $n\in\omega$, the relative algebra $\mathcal{C}\upharpoonright h(e_n)$ must be isomorphic to $Int(\omega^{\beta} \cdot \ell_n)$ for some finite $\ell_n \geq 1+f(n)$. Thus, the function $\xi(n) := \ell_n$ is computable in $h$, and $\xi$ dominates $f$. Therefore, we deduce $h \geq_T \xi \geq_T \mathbf{0}^{\langle 2\alpha\rangle}$.


	\textsc{Case II.}
	 Let $\alpha$ be a limit ordinal. For simplicity, we give the proof for the case when $\mathcal{B}\cong Int(\omega^{\alpha})$. A general case of $Int(\omega^{\alpha}\cdot m)$ is treated in a similar way.

	 Consider the standard fundamental sequence $(\beta_i)_{i\in\omega}$ for $\alpha$. Without loss of generality, we may assume that for every $i\in\omega$, $0<\beta_i <\beta_{i+1}$ and
	 \[
	 	\beta_i=\begin{cases}
	 		2\gamma_i, & \text{if } \beta_i <\omega,\\
	 		2\gamma_i + 1, & \text{if } \beta_i \geq \omega.
	 	\end{cases}
	 \]

	 Fix a total function $f_{[\alpha]}$ from Proposition~\ref{prop:good-lm}. Recall that for any $i$ and $x$, we have $f_{[\alpha]}(\langle i,x\rangle) = f_{[\beta_i]}(x)$, where the functions $f_{[\beta_i]}$ are $\mathbf{0}^{\langle 2\gamma_i\rangle}$-limit\-wise monotonic, uniformly in $i\in\omega$.

	 Note the following: If $\lambda \geq \gamma$ and $\kappa \geq \gamma$, then $\omega^{\lambda} \equiv_{2\gamma} \omega^{\kappa}$ (see, e.g., Lemma~15.9 in \cite{AK00}). Hence, for every $i\in\omega$, we have:
	 \[
	 	 \omega^{\gamma_i} \geq_{2\gamma_i} \omega^{\gamma_{i+1}} \geq_{2\gamma_i} \omega^{\gamma_{i+2}} \geq_{2\gamma_i} \dots \geq_{2\gamma_i} \omega^{\gamma_{i+j}} \geq_{2\gamma_i} \dots .
	 \]
	 From now on, for the sake of convenience, we use the notation $\gamma[i] := \gamma_i$.

	 We employ Theorem~\ref{theo:uniform} with the following parameters:
	 \begin{itemize}
	 	\item the sequence $\alpha_i := 2\gamma_i$;

	 	\item the $\alpha$-friendly family consisting of $\mathcal{A}^i_k := \omega^{\gamma[i+k]}$; and

	 	\item the functions $f_i := f_{[\beta_i]}$.
	\end{itemize}
	This gives us a computable sequence of linear orders:
	 \[
	 	\mathcal{L}_{\langle i,k\rangle} \cong \omega^{\gamma[i+f_{[\beta_i]}(k)]}.
	 \]

	 We define a Boolean algebra
	 \[
	 	\mathcal{A} := \sum_{n\in\omega} Int(\mathcal{L}_n).
	\]
	It is not difficult to show that $\mathcal{A}$ is a computable isomorphic copy of $Int(\omega^{\alpha})$. For $n\in\omega$, let $e(n)$ denote the element induced by $\mathbf{1}^{Int(\mathcal{L}_n)}$ inside $\mathcal{A}$.

	Suppose that $h$ is an isomorphic embedding from $\mathcal{A}$ into $\mathcal{C}$. We define a total function $\xi$~--- for numbers $i$ and $k$, the value $\xi(\langle i,k\rangle)$ is computed as follows.
	\begin{enumerate}
		\item Find the ordinal $\delta <\alpha$ such that the relative algebra $\mathcal{C}\upharpoonright h(e(\langle i,k\rangle))$ is isomorphic to $Int(\omega^{\delta}\cdot \ell)$ for some $\ell \in \omega$. Clearly, $\delta\geq \gamma[i+f_{[\beta_i]}(k)]$.

		\item Define the value $\xi(\langle i,k\rangle)$ as the least number $j$ such that $\gamma[j] >\delta$.
	\end{enumerate}
	Note that $\xi(\langle i,k\rangle) > i + f_{[\beta_i]}(k) = i+f_{[\alpha]}(\langle i,k\rangle)$. Thus, the function $\xi$ dominates $f_{[\alpha]}$. Hence, by Proposition~\ref{prop:good-lm}, we deduce that $h\geq_T \xi \geq_T \mathbf{0}^{(\alpha)}$.
	This concludes the proof of Theorem~\ref{theo:BA}.
	\qed


\subsection{Proof of Corollary~\ref{corollary:gen-BA}} \label{subsect:coroll}

	If $\mathcal{B}$ is superatomic, then by Theorem~\ref{theo:BA}, $\mathcal{B}$ satisfies the first condition. Thus, we will assume that $\mathcal{B}$ is not superatomic, i.e. $\mathcal{B}$ is bi-embeddable with the atomless Boolean algebra $Int(1+\eta)$.

	Let $\mathcal{A}$ be a computable copy of $Int(1+\eta)$. Fix a computable linear order $\mathcal{L}$ isomorphic to $\omega^{CK}_1\cdot ( 1 + \eta) $ such that $\mathcal{L}$ has no hyperarithmetic descending chains. Set $\mathcal{C} := Int(\mathcal{L})$. Clearly, both $\mathcal{A}$ and $\mathcal{C}$ are bi-embeddable with $\mathcal{B}$.

	Towards a contradiction, assume that there is a hyperarithmetic embedding $h$ from $\mathcal{A}$ into $\mathcal{C}$. Choose a non-zero element $a\in\mathcal{A}$ such that $h(a)$ can be represented in the following form:
	\begin{equation}\label{equ:presentation}
		[u_0(a),v_0(a))_{\mathcal{L}} \cup [u_1(a), v_1(a))_{\mathcal{L}} \cup \dots \cup [u_{k(a)}(a),v_{k(a)}(a))_{\mathcal{L}},
	\end{equation}
	where $u_0(a) <_{\mathcal{L}} v_0(a) <_{\mathcal{L}} u_1(a) <_{\mathcal{L}} v_1(a)<_{\mathcal{L}}\dots <_{\mathcal{L}}u_{k(a)}(a)<_{\mathcal{L}} v_{k(a)}(a)$. Set $b_0:=a$ and $w_0 := v_{k(a)}(a)$.

	Now suppose that the elements $b_0,b_1,\dots,b_n \in\mathcal{A}$ and $w_0 >_{\mathcal{L}} w_1 >_{\mathcal{L}} \dots >_{\mathcal{L}} w_n$ are already defined. Choose arbitrary non-zero $c$ and $d$ from $\mathcal{A}$ with $c\cap d = 0^{\mathcal{A}}$ and $c\cup d = b_n$. For the elements $h(c)$ and $h(d)$, consider their representations of the form~(\ref{equ:presentation}). Clearly, we have either $v_{k(c)}(c) <_{\mathcal{L}} w_n$ or $v_{k(d)}(d) <_{\mathcal{L}} w_n$. Choose $w_{n+1}$ as the element from $\{ v_{k(c)}(c), v_{k(d)}(d)\}$ which is strictly less than $w_n$ in $\mathcal{L}$.

	Then the constructed sequence $(w_n)_{n\in\omega}$ is hyperarithmetic and strictly descending, which contradicts the choice of $\mathcal{L}$. Therefore, the structure $\mathcal{B}$ is not hyperaritmetically b.e.\ categorical. By Theorem~3.1 of~\cite{bazhenov2020}, this implies that $\mathcal{B}$ has no degree of b.e.\ categoricity.
\qed


\subsection{Proof of Theorem~\ref{theo:uniform}} \label{subsect:proof-uniform}

The proof employs the $\alpha$-systems technique. For a detailed exposition of this method, the reader is referred to~\cite{AK00}. Here we give only the necessary definitions and results.

For sets $L$ and $U$, an \emph{alternating tree} on $L$ and $U$ is a tree $P$ consisting of non-empty finite sequences $\sigma = l_0 u_1 l_1 u_2 \dots$, where $l_i \in L$ and $u_j \in U$. We assume that every $\sigma \in P$ has a proper extension in $P$.

An \emph{instruction function} for $P$ is a function $q$ from the set of sequences in $P$ of odd length (i.e. those with last term in $L$) to $U$, such that if $q(\sigma) = u$, then $\sigma u \in P$. A \emph{run} of $(P,q)$ is a path $\pi = l_0 u_1 l_1 u_2 \dots$ through $P$ such that
\[
	u_{n+1} = q(l_0 u_1 l_1 u_2 \dots u_n l_n)
\]
for every $n\in\omega$. An \emph{enumeration function} on $L$ is a function $E$ from $L$ to the set of all finite subsets of $\omega$. If $\pi = l_0 u_1 l_1 u_2\dots$ is a path through $P$, then $E(\pi) = \bigcup_{i\in\omega} E(l_i)$.


Suppose that $L$ and $U$ are c.e. sets, $E$ is a partial computable enumeration function on $L$, and $P$ is a c.e. alternating tree on $L$ and $U$. We assume that all $\sigma \in P$ start with the same symbol $\hat l \in L$.

Fix a non-zero computable ordinal $\alpha$. Suppose that $\leq_{\beta}$, $\beta<\alpha$, are binary relations on $L$ such that $\leq_{\beta}$ are c.e., uniformly in $\beta < \alpha$.

The structure
\[
	(L,U,\hat l, P,E, (\leq_{\beta})_{\beta<\alpha})
\]
is an \emph{$\alpha$-system} if is satisfies the following conditions:
\begin{enumerate}
	\item $\leq_{\beta}$ is reflexive and transitive, for $\beta<\alpha$;

	\item $(l\leq_{\gamma} l') \Rightarrow (l \leq_{\beta} l')$, for $\beta < \gamma < \alpha$;

	\item if $l \leq_0 l'$, then $E(l) \subseteq E(l')$;

	\item if $\sigma u \in P$, where $\sigma$ ends in $l^0\in L$; and
	\[
		l^0 \leq_{\beta_0} l^1 \leq_{\beta_1} \dots \leq_{\beta_{k-1}} l^k,
	\]
	for $\alpha > \beta_0 > \beta_1 >\dots >\beta_{k}$, then there exists $l^{\ast} \in L$ such that $\sigma u l^{\ast} \in P$, and $l^i \leq_{\beta_i} l^{\ast}$ for all $i\leq k$.
\end{enumerate}

\begin{theorem}[Ash and Knight {\cite[Theorem 14.1]{AK00}}] \label{th:app}
	Let $(L,U,\hat l, P, E, (\leq_{\beta})_{\beta<\alpha})$ be an $\alpha$-system.
	Then for any $\Delta^0_{\alpha}$ instruction function $q$, there is a run $\pi$ of $(P,q)$ such that $E(\pi)$ is c.e., while $\pi$ itself is $\Delta^0_{\alpha}$. Moreover, from a $\Delta^0_{\alpha}$ index for $q$, together with a computable sequence of c.e. and computable indices for the components of the $\alpha$-system, we can effectively determine a $\Delta^0_{\alpha}$ index for $\pi$ and a c.e. index for $E(\pi)$.
\end{theorem}

Now we are ready to give the proof of Theorem~\ref{theo:uniform}. Essentially, the construction described below extends the proof of Theorem~18.9 from~\cite{AK00}.


Fix an index $e$ such that for each $i\in\omega$, the function $\Phi_e^{\emptyset^{\langle \alpha_i\rangle}}(i,\cdot,\cdot)$ approximates $f_i$ in a limitwise monotonic way. Set:
\[
	g(i,x,s) := \Phi_e^{\emptyset^{\langle \alpha_i\rangle}}(i,x,s).
\]
Note that $f_i(x) = \lim_s g(i,x,s)$.

Given two indices $M,n\in\omega$, we define (in a uniform way) the following objects:
\begin{itemize}
	\item[(a)] an $\alpha_M$-system $\textbf{S}_M = (L_M, U_M, \hat l_M, P_M, E_M, (\leq_{\beta,M})_{\beta < \alpha_M})$ (the same for all $n\in\omega$), and

	\item[(b)] a $\Delta^0_{\alpha_M}$ instruction function $q_{M,n}$, with $\Delta^0_{\alpha_M}$ index that can be computed effectively from $M$ and $n$.
\end{itemize}
The desired structure $\mathcal{C}^M_n$ is obtained from a run of $(P_M, q_{M,n})$. The uniformity of Theorem~\ref{th:app} guarantees that the structures $\mathcal{C}^M_n$ are uniformly computable.

We focus on a given pair $(M,n)$, hence, we will slightly abuse the notations, and omit the subscript $M$ in the names of our objects.

Assume that $C$ is an infinite computable set of constants, for the universe of $\mathcal{C}^{M}_n$. We also assume that all the structures $\mathcal{A}^M_k$, $k\in\omega$, have the same universe $A$. Let $\mathcal{F}$ be the set of finite partial 1-1 functions $p$ from $C$ to $A$. Let $U = \omega$ and let $L$ consist of the pairs in $\omega \times \mathcal{F}$, and one extra
element $\hat l := (-1, \emptyset)$. If $(k,p) \in \omega \times \mathcal{F}$, then $(k,p)$ represents $(\mathcal{A}^M_k,p)$.

The \emph{standard enumeration function} $E^{st}$ is defined as follows. Assume that $m\in \omega$, and $\bar x$ is the sequence of the first $m$ variables. For an $m$-tuple $\bar a$ in an $L$-structure $\mathcal{B}$, the set $E^{st}(\mathcal{B},\bar a)$ consists of the basic formulas $\psi(\bar x)$, with G\"odel number less than $m$, such that $\mathcal{B} \models \psi(\bar a)$.

Following~\cite{AK00}, we extend the standard enumeration function and the standard back-and-forth relations on pairs $(\mathcal{A}^M_k,\bar a)$ to tuples $(k,p) \in L$. If $l = (k,p)$, where $k\in\omega$ and $p$ maps $\bar b$ to $\bar a$, then $E(l) = E^{st}(\mathcal{A}^M_k, \bar a)$. We set $E(\hat l) = \emptyset$.

Assume that $\beta <\alpha_M$. If $l = (k,p)$ and $l' = (j,q)$, where $p$ maps $\bar b$ to $\bar a$ and $q$ maps $\bar b'$ to $\bar a'$, then $l \leq_{\beta} l'$ if and only if $\bar b \subseteq \bar b'$ and $(\mathcal{A}^M_k, \bar a) \leq_{\beta} (\mathcal{A}^M_j, \bar a')$. We write $l \subseteq l'$ if $k = j$ and $p \subseteq q$. For all $l\in L$, we let $\hat l \leq_{\beta} l$ and $\hat l \subseteq l$.

The tree $P$ consists of the finite sequences $\sigma = \hat l u_1 l_1 u_2 \dots$ such that $u_k \in U$, $l_k\in L$, and the following conditions hold:
\begin{enumerate}
	\item $l_k$ has the form $(u_k, p_k)$;

	\item $dom(p_k)$ and $ran(p_k)$ include the first $k$ elements of the sets $C$ and $A$, respectively;

	\item $u_1 = 0$;

	\item if $u_{k+1} \neq u_k$, then $u_{k+1} = u_k + 1$;

	\item if $u_k = u_{k+1}$, then $l_k \subseteq l_{k+1}$, and in any case, $l_k \leq_0 l_{k+1}$.
\end{enumerate}
We have described all ingredients of our $\alpha_M$-system. Note that the description is uniform in $M\in\omega$.

\begin{lemma}
	$(L, U, \hat l, P, E, (\leq_{\beta})_{\beta < \alpha_M})$ is an $\alpha_M$-system.
\end{lemma}
\begin{proof}
	It is easy to check that the sets $L,U,P$ are c.e. and the function $E$ is partial computable. Since the family $\{\mathcal{A}^i_k\,\colon i,k\in\omega\}$ is $\alpha$-friendly, the relations $\leq_{\beta}$ are c.e. uniformly in $\beta < \alpha_M$ (note that this fact is true even for all $\beta < \alpha$, but we will not use this). It is not hard to verify the first three conditions from the definition of an $\alpha_M$-system.

	Assume that $\sigma u\in P$, where $\sigma$ ends in $l^0$, and
	\[
		l^0 \leq_{\beta_0} l^1 \leq_{\beta_1} \dots \leq_{\beta_{k-1}} l^k,
	\]
	for $\alpha_M > \beta_0 > \dots >\beta_k$. We want to find $l^{\ast}\in L$ such that $\sigma u l^{\ast}\in P$ and $l^i \leq_{\beta_i} l^{\ast}$ for all $i\leq k$. Suppose that $l^i = (v_i,p_i)$. Using the properties of the standard back-and-forth relations (see \cite[\S 15.1]{AK00}), one can find functions $p^{\ast}_i \supseteq p_i$, $i\leq k$, such that $p^{\ast}_k = p_k$ and $(v_i, p^{\ast}_i) \leq_{\beta_{i}} (v_{i-1}, p^{\ast}_{i-1})$, for $0<i\leq k$.  Note that $l^i \leq_{\beta_i} (v_0, p^{\ast}_0)$ for all $i\leq k$.

	If $u = v_0$, then we build an extension $q\supseteq p^{\ast}_0$ (to include the necessary elements in the domain and range of our function) and obtain the desired $l^{\ast} = (v_0, q)$. Now assume that $u=v_0 + 1$. Since $\mathcal{A}^M_{v_0 + 1} \leq_{\alpha_M} \mathcal{A}^M_{v_0}$, there is a function $q_0$ such that $(v_0,p^{\ast}_0) \leq_{\beta_0} (v_0+1, q_0)$. Again, we find a suitable extension $q\supseteq q_0$ and produce the desired $l^{\ast} = (v_0+1, q)$.
\end{proof}

The instruction function $q_{M,n}$ is defined as follows. We set $q_{M,n}(\hat l) = 0$. For an element $\sigma = \hat l u_1 l_1 \dots u_s l_s$ in $P$, let
\[
	q_{M,n}(\sigma) = \left\{
		\begin{array}{ll}
			u_{s}+1, & \text{~if~} g(M,n,s) > u_s,\\
			u_s, & \text{~otherwise}.
		\end{array}
	\right.
\]
Given numbers $M,n$ and a $\Delta^0_{\alpha_M}$ index for the function $g(M,\cdot,\cdot)$, we can compute a $\Delta^0_{\alpha_M}$ index for $q_{M,n}$.

We apply Theorem~\ref{th:app}, and for each $M$ and $n$, we obtain a run $\pi_{M,n} = \hat l u^{M,n}_1 l^{M,n}_1 u^{M,n}_2 \dots$ of $(P_M, q_{M,n})$ such that $E(\pi_{M,n})$ is c.e., uniformly in $M,n$. Suppose that $l^{M,n}_s = (u^n_s, p^n_s)$.

Recall that $g(M,n,s) \leq g(M,n,s+1)$ and $\lim_s g(M,n,s) = f_M(n)$. Hence, it is easy to show that for all $s$, we have $u^n_s \leq g(M,n,s)$. Moreover, there is a stage $s^{\ast}$ such that $u^n_s = f_M(n)$ for all $s \geq s^{\ast}$. The mapping $F^{-1} = \bigcup_{s \geq s^{\ast}} p^n_s$ is a 1-1 function from $C$ onto $\mathcal{A}^M_{f_M(n)}$. The map $F$ induces a structure $\mathcal{C}^M_n$ on the universe $C$ such that $D(\mathcal{C}^M_n) = E(\pi_{M,n})$.

So, we constructed a uniformly computable sequence $\{\mathcal{C}^M_n\}_{M,n\in\omega}$ such that $\mathcal{C}^M_n \cong \mathcal{A}^M_{f_M(n)}$ for all $M$ and $n$. Theorem~\ref{theo:uniform} is proved.
\qed

\bibliographystyle{alpha}
\bibliography{References}

\begin{thebibliography}{BDKM19}

\bibitem[AK90]{AK90}
C.~J. Ash and J.~F. Knight.
\newblock Pairs of recursive structures.
\newblock {\em Ann. Pure Appl. Logic}, 46(3):211--234, 1990.

\bibitem[AK00]{AK00}
C.~J. Ash and J.~F. Knight.
\newblock {\em Computable structures and the hyperarithmetical hierarchy},
  volume 144 of {\em Studies in Logic and the Foundations of Mathematics}.
\newblock Elsevier Science B.V., Amsterdam etc., 2000.

\bibitem[AKMS89]{AKMS-89}
Chris Ash, Julia Knight, Mark Manasse, and Theodore Slaman.
\newblock Generic copies of countable structures.
\newblock {\em Ann. Pure Appl. Logic}, 42(3):195--205, 1989.

\bibitem[AR]{alvir2018}
Rachael Alvir and Dino Rossegger.
\newblock The complexity of {{Scott}} sentences of scattered linear orderings.
\newblock {\em submitted for publication}.

\bibitem[Ash86]{Ash86}
C.~J. Ash.
\newblock Recursive labelling systems and stability of recursive structures in
  hyperarithmetical degrees.
\newblock {\em Trans. Amer. Math. Soc.}, 298(2):497--514, 1986.

\bibitem[Ash87]{Ash87}
C.~J. Ash.
\newblock Categoricity in hyperarithmetical degrees.
\newblock {\em Ann. Pure Appl. Logic}, 34(1):1--14, 1987.

\bibitem[Ash98]{Ash-98}
Chris~J. Ash.
\newblock Isomorphic recursive structures.
\newblock In {Yu.}~L. Ershov, S.~S. Goncharov, A.~Nerode, and J.~B. Remmel,
  editors, {\em Handbook of recursive mathematics, {V}ol.\ 1}, volume 138 of
  {\em Stud. Logic Found. Math.}, pages 167--181. North-Holland, Amsterdam,
  1998.

\bibitem[Baz13]{Bazh-13}
Nikolay~A. Bazhenov.
\newblock Degrees of categoricity for superatomic {B}oolean algebras.
\newblock {\em Algebra and Logic}, 52(3):179--187, 2013.

\bibitem[Baz16]{Bazh16}
Nikolay~A. Bazhenov.
\newblock Degrees of autostability relative to strong constructivizations for
  {B}oolean algebras.
\newblock {\em Algebra Logic}, 55(2):87--102, 2016.

\bibitem[Baz17]{Bazh17}
Nikolay Bazhenov.
\newblock Turing computable embeddings, computable infinitary equivalence, and
  linear orders.
\newblock In J.~Kari, F.~Manea, and I.~Petre, editors, {\em Unveiling Dynamics
  and Complexity}, volume 10307 of {\em Lect. Notes Comput. Sci.}, pages
  141--151. Springer, Cham, 2017.

\bibitem[BDKM19]{BDKM-19}
Nikolay Bazhenov, Rod Downey, Iskander Kalimullin, and Alexander Melnikov.
\newblock Foundations of online structure theory.
\newblock {\em Bull. Symb. Log.}, 25(2):141--181, 2019.

\bibitem[BFRS]{bazhenov2020}
Nikolay Bazhenov, Ekaterina Fokina, Dino Rossegger, and Luca {San Mauro}.
\newblock Degrees of bi-embeddable categoricity.
\newblock {\em to appear in Computability}.

\bibitem[BFRS19]{bazhenov2018a}
Nikolay Bazhenov, Ekaterina Fokina, Dino Rossegger, and Luca {San Mauro}.
\newblock Degrees of bi-embeddable categoricity of equivalence structures.
\newblock {\em Archive for Mathematical Logic}, 58(5--6):543--563, 2019.

\bibitem[CFS13]{CFS-2013}
Barbara~F. Csima, Johanna N.~Y. Franklin, and Richard~A. Shore.
\newblock Degrees of categoricity and the hyperarithmetic hierarchy.
\newblock {\em Notre Dame J. Formal Logic}, 54(2):215--231, 2013.

\bibitem[Chi90]{Chi-90}
John Chisholm.
\newblock Effective model theory vs. recursive model theory.
\newblock {\em J. Symb. Log.}, 55(3):1168--1191, 1990.

\bibitem[CR98]{CR-98}
Douglas Cenzer and Jeffrey~B. Remmel.
\newblock Complexity theoretic model theory and algebra.
\newblock In {Yu.}~L. Ershov, S.~S. Goncharov, A.~Nerode, and J.~B. Remmel,
  editors, {\em Handbook of recursive mathematics, {V}ol.\ 1}, volume 138 of
  {\em Stud. Logic Found. Math.}, pages 381--513. North-Holland, Amsterdam,
  1998.

\bibitem[FHM14]{FHM-14}
Ekaterina~B. Fokina, Valentina Harizanov, and Alexander Melnikov.
\newblock Computable model theory.
\newblock In R.~Downey, editor, {\em Turing's legacy: Developments from
  Turing's ideas in logic}, volume~42 of {\em Lect. Notes Logic}, pages
  124--194. Cambridge University Press, Cambridge, 2014.

\bibitem[FKM10]{fokina2010}
Ekaterina~B. Fokina, Iskander Kalimullin, and Russell Miller.
\newblock Degrees of categoricity of computable structures.
\newblock {\em Archive for Mathematical Logic}, 49(1):51--67, 2010.

\bibitem[FRSM19]{fokina2019a}
Ekaterina Fokina, Dino Rossegger, and Luca San~Mauro.
\newblock Bi-embeddability spectra and bases of spectra.
\newblock {\em Mathematical Logic Quarterly}, 65(2):228--236, 2019.

\bibitem[GM08]{GreenMon}
Noam Greenberg and Antonio Montalb{\'a}n.
\newblock Ranked structures and arithmetic transfinite recursion.
\newblock {\em Trans. Am. Math. Soc.}, 360(3):1265--1307, 2008.

\bibitem[Gon73]{Gon-73}
Sergei~S. Goncharov.
\newblock Constructivizability of superatomic {B}oolean algebras.
\newblock {\em Algebra and Logic}, 12(1):17--22, 1973.

\bibitem[Gon97]{Gon-97}
Sergei~S. Goncharov.
\newblock {\em Countable {B}oolean algebras and decidability}.
\newblock Consultants Bureau, New York, 1997.

\bibitem[Jul]{jullien1969}
Pierre Jullien.
\newblock Contribution \`a l'\'etude des types d'ordres dispers\'es.

\bibitem[Kop89]{Koppelberg}
Sabine Koppelberg.
\newblock {\em Handbook of {B}oolean algebras. {V}ol. 1}.
\newblock North-Holland, Amsterdam, 1989.
\newblock edited by J. D. Monk and R. Bonnet.

\bibitem[Mal62]{Mal62}
Anatoli{\"i}~I. Mal'tsev.
\newblock On recursive abelian groups.
\newblock {\em Sov. Math. Dokl.}, 3:1431--1434, 1962.

\bibitem[Mel17]{Mel-17}
Alexander~G. Melnikov.
\newblock Eliminating unbounded search in computable algebra.
\newblock In J.~Kari, F.~Manea, and I.~Petre, editors, {\em Unveiling Dynamics
  and Complexity}, volume 10307 of {\em LNCS}, pages 77--87. Springer, Cham,
  2017.

\bibitem[Mon05]{montalban2005}
Antonio Montalb\'an.
\newblock Up to equimorphism, hyperarithmetic is recursive.
\newblock {\em The Journal of Symbolic Logic}, 70(2):360--378, 2005.

\bibitem[Mon06]{montalban2006}
Antonio Montalb\'an.
\newblock Equivalence between {{Fra\"iss\'e}}'s conjecture and {{Jullien}}'s
  theorem.
\newblock {\em Annals of Pure and Applied Logic}, 139(1-3):1--42, 2006.

\bibitem[Soa16]{Soare-New}
Robert~I. Soare.
\newblock {\em Turing computability. Theory and applications}.
\newblock Springer, Berlin, 2016.

\bibitem[Spe55]{spector1955}
Clifford Spector.
\newblock Recursive well-orderings.
\newblock {\em The Journal of Symbolic Logic}, 20(2):151--163, 1955.

\end{thebibliography}

\end{document}